\documentclass[11pt]{amsart}

\usepackage{bm}
\usepackage{fullpage}
\usepackage{amssymb}
\usepackage{amsmath,amsfonts,amsthm}
\usepackage{hyperref}
\usepackage{color}
\usepackage{cite}

\newtheorem{alphtheorem}{Theorem}

\newtheorem{alphlemma}{Lemma}

\newtheorem{question}{Question}
\newtheorem{lemma}{Lemma}

\newtheorem{theorem}{Theorem}
\newtheorem{corollary}{Corollary}

\theoremstyle{definition}
\newtheorem{definition}{Definition}

\DeclareMathOperator\St{\mathcal{S}}

\DeclareMathOperator\J{\mathcal{J}}
\DeclareMathOperator\N{\mathcal{N}}

\DeclareMathOperator\LL{\mathcal{L}}

\DeclareMathOperator\F{\mathcal{F}}
\DeclareMathOperator\HM{\mathcal{HM}}

\DeclareMathOperator\B{\mathcal{B}}

\DeclareMathOperator\KG{\operatorname{KG}}

\def\isdef{\mbox {$\ \stackrel{\rm def}{=} \ $}}

\title{ size and structure of large $(s,t)$-union intersecting families }

\author{Ali Taherkhani}
\address{
Department of Mathematics, Institute for Advanced Studies in Basic Sciences (IASBS), Zanjan 45137-66731, Iran}
\address{School of Mathematics, Institute for Research in Fundamental Sciences (IPM), P.O. Box 19395-5746, Tehran, Iran}\email{ali.taherkhani@iasbs.ac.ir}
\begin{document}
\maketitle 

\begin{abstract}
A  family $\F$ of sets is said to be intersecting  if
any two sets in $\F$  have nonempty intersection. The celebrated Erd{\H o}s-Ko-Rado theorem determines the size and structure of the largest intersecting  family of $k$-sets on an $n$-set $X$.  Also, the Hilton-Milner theorem determines the size and structure of the second largest intersecting family of $k$-sets.
An $(s,t)$-union intersecting family is a family of $k$-sets on an $n$-set $X$ such that
for any $A_1,\ldots,A_{s+t}$ in this family,
$\left(\cup_{i=1}^s A_i\right)\cap\left(\cup_{i=1}^t A_{i+s}\right)\neq \varnothing.$ 
Let $\ell(\F)$ be the 
 minimum number of sets in $\F$ such that by removing them  the resulting subfamily
is intersecting. 
In this paper, for $t\geq s\geq 1$ and sufficiently large $n$,
we characterize the size  and structure of  $(s,t)$-union intersecting families with  maximum possible size and $\ell(\F)\geq s+\beta$, where $\beta$ is a nonnegative integer. 
This allows  us to find out the size  and structure  of some large and maximal $(s,t)$-union intersecting families.
 Our results are nontrivial extensions of some recent generalizations of the Erd{\H o}s-Ko-Rado theorem such as 
  the Han and Kohayakawa theorem~[Proc. Amer. Math. Soc. 145 (2017), pp. 73--87]
 which finds the structure of the third largest intersecting family,  
 the Kostochka  and Mubayi theorem~[Proc. Amer. Math. Soc. 145 (2017), pp. 2311--2321], and the more recent 
 Kupavskii's theorem [arXiv:1810.009202018 (2018)] whose both results determine
  the size and structure of the $i$th largest intersecting family of $k$-sets for 
 $i\leq k+1$.  In particular, when $s=1$, 
  we prove that a  Hilton-Milner-type stability theorem holds for $(1,t)$-union intersecting families, that
indeed, confirms a conjecture of Alishahi and Taherkhani~[J. Combin. Theory Ser. A 159 (2018), pp. 269--282]. 

\noindent As the induced subgraph on an $(s,t)$-union intersecting family in the Kneser graph $\KG_{n,k}$ is a  $K_{s,t}$-free subgraph, we can  extend our results to
$K_{s_1,\ldots,s_{r+1}}$-free subgraphs of Kneser graphs. In fact, when $n$ is sufficiently large, we  characterize the
size  and structure  of large and maximal $K_{s_1,\ldots,s_{r+1}}$-free subgraphs of Kneser graphs. In particular, when $s_1=\cdots=s_{r+1}=1$ our result provides some stability results  related to  the famous Erd{\H o}s matching conjecture.
\\ \\
\end{abstract}

\section{Introduction and Main Result}
\subsection{Erd{\H o}s-Ko-Rado theorem and its generalization}
Let $n$ and $k$ be two positive integers such that $n\geq k$. The symbol $[n]$ stands for the set
$\{1,\ldots,n\}$ and the symbol $[k,n]$ stands for the set $[n]\setminus [k-1]$. The family of all $k$-element subsets (or $k$-sets) of $[n]$ is denoted by $[n]\choose k$.  In this paper, we only consider  families  which 
consist of  $k$-sets on $[n]$.
 A family $\F$ is said to be {\it intersecting} if  the intersection 
of every two members of $\F$ is non-empty. If all members of $\F$ contain a fixed element of $[n]$, then it is clear that $\F$ is an intersecting family
 which is called 
 a {\it star} or a {\it trivial} family. For each $i\in[n]$, the family $\St_i\isdef\{A\in{[n]\choose k}|i\in A \}$ 
is a maximal star. Also, the following two families are well-known examples for intersecting families.
 Let $B$ be a $k$-set of $[n]$ such that $1\not\in B$. Define
$$\HM \isdef\{A|\,1\in A,\,\, A\cap B\neq\varnothing\} \cup \{B\}$$
and
$$\HM' \isdef\{A|\,|A\cap\{1,2,3\}|\geq 2\}.$$
Note that for $2\le k\leq 3$, we have $|\HM|=|\HM'|$  and if  $n>2k$ and $k\geq4$, then $|\HM|>|\HM'|$.

The well-known Erd{\H o}s-Ko-Rado theorem~\cite{EKR61} states that every intersecting family of ${[n]\choose k}$ has  cardinality
at most ${n-1\choose k-1}$ provided that $n\geq 2k$; moreover, if $n>2k$, then  the only intersecting families  of this cardinality are maximal stars.

As a generalization of the Erd{\H o}s-Ko-Rado theorem, Hilton and Milner~\cite{HilMil67} proved a useful and interesting stability  result. 
They showed that for $n>2k$ the maximum possible  size of a nontrivial intersecting family $\F$ of ${[n]\choose k}$  is ${n-1\choose k-1}-{n-k-1\choose k-1}+1$. Furthermore, equality is possible only for a family $\F$ which is isomorphic to $\HM$ or $\HM'$, the latter can hold only  for $k\leq 3$.

 A family $\F$ is called a {\it Hilton-Milner family} if $\F$ is isomorphic to a subfamily of $\HM$ for some $k$  or   it is isomorphic to a subfamily of  $\HM'$  for $k\in\{2,3\}$.

There  also exist some other interesting extensions of Erd{\H o}s-Ko-Rado and Hilton-Milner theorems in the literature ( e.g.\cite{MA,MR3482268,MR3403515,DeFr1983,MR0480051,MR1415313,Frankl13,FRANKL20121388,FranLucMie12,MR3022158,GR,MR2489272,HK,KATONA1972183,KoMu,MR2202076,Kup,MR2285800,MR771733}).

The Erd{\H o}s-Ko-Rado theorem  determines the maximum size of an intersecting family of $k$-sets on $[n]$ and the Hilton-Milner theorem shows that 
a nontrivial intersecting family has  cardinality at most ${n-1\choose k-1}-{n-k-1\choose k-1}+1.$
 Beyond Hilton-Milner theorem, it was shown by Hilton and Milner~\cite{HilMil67} that the maximum size of a nontrivial intersecting family which is not a   Hilton-Milner family is at most 
${n-1\choose k-1}-{n-k-1\choose k-1}-{n-k-2\choose k-2}+2$. In fact they proved the following interesting result (see \cite{HK,HilMil67}).
\begin{alphtheorem}{\rm \cite{HilMil67}}\label{HM}
Let $n,k$, and $s$ be positive integers with $\min\{3,s\}\leq k\leq{n\over 2}$ and let $\F=\{A_1,\ldots,A_m\}$ be an intersecting family on $[n]$.
If for any $S\subset [m]$
with $|S|>m-s$, we have $\cap_{i\in S}A_i=\varnothing$, then
\begin{equation}\label{eq:hm}
m\leq \left\{\begin{array}{lll} 
{n-1\choose k-1}-{n-k\choose k-1}+n-k&&\,\, {\rm if}\,\, 2<k\le s+2,\\
&&\\
{n-1\choose k-1}-{n-k\choose k-1}+{n-k-s\choose k-s-1}+s&&\,\, {\rm if}\,\, k\leq 2\,\, {\rm or}\,\, k\geq s+2.
       \end{array}
 \right.
\end{equation}
Moreover, the bounds in Inequality~$\rm({ \ref{eq:hm}})$ are the best possible.
\end{alphtheorem}
Recently, Han and Kohayakawa  gave a different and simpler proof of  Theorem~\ref{HM}. 
Moreover, they characterized all extremal families achieving the bounds in (\ref{eq:hm}) (for more details see~{\cite{HK}}).
In this regard they introduced the following  construction.
\begin{definition}\label{def1}
Let $i$ be a nonnegative integer. For any $(i+1)$-set $J\subset [n]$ with $1\in J$ and any $(k-1)$-set $E\subset [n]\setminus J$, 
define the family $\J_{i}$  as follows,
 $$\J_{i}\isdef\{A: E\subset A,\,\,A\cap J\neq \varnothing\}\cup\{A:J\subset A\}\cup\{A:1\in A,\,\, A\cap E\neq \varnothing\}.$$
\end{definition}
  Note that $\J_0=\St_{1}$, $\J_{1}=\HM$,  $|\J_{i}|={n-1\choose k-1}-{n-k\choose k-1}+{n-k-i\choose k-i-1}+i$, and $|\J_i\setminus \St_1|=i$.
\begin{alphtheorem}{\rm\cite{HK}\label{HK}}
 Let $n,k$  be positive integers with $3\leq k<{n\over 2}$ and let $\F$ be an intersecting family of $k$-sets on $[n]$.
Assume that $\F$ is neither a  star nor a Hilton-Milner family. Then $|\F|\leq |\J_2|$.
Moreover, for $k\geq 5$, equality holds if and only if $\F$ is isomorphic to $\J_2$.
\end{alphtheorem}
 \begin{definition}
For $i\leq k$ let us define
 the family $\F_{i}$ of ${[n]\choose k}$ as follows,
$$\F_i\isdef [2,k+1]\cup[i+1,k+i]\cup \{A: 1\in A, A\cap[2,k+i]=\varnothing\}.$$
\end{definition}

In \cite{KoMu}, Kostochka and Mubayi proved that the size of an intersecting family which is neither a star nor is contained in 
$\J_i$, for $i\in\{1,\ldots,k-1, n-k\}$, is  at  most $|\F_3|$ for $k\geq 5$ and sufficiently large $n=n(k)$. 
Also, more recently Kupavskii \cite{Kup} extended this result and  showed that the same result holds when $5\leq k<{n\over 2}$.
 \begin{alphtheorem}{\rm\cite{KoMu,Kup}\label{Kup}}
Let $n,k$  be positive integers with $5\leq k<{n\over 2}$ and let $\F$ be an intersecting family of $k$-sets on $[n]$ with $|\F|>|\F_3|$. 
Then $\F\subseteq \J_i$ for $i\in\{0,1,\ldots k-1,n-k\}$.
\end{alphtheorem}


 \subsection{  $G$-free subgraphs  of Kneser graphs and $(s,t)$-union intersecting families}

 Let $n\geq 2k$. The {\it Kneser graph $\KG_{n,k}$} is a graph whose vertex set is ${[n]\choose k}$ 
where  two vertices are adjacent if their corresponding sets are disjoint.  
From another point of view, the  Erd{\H o}s-Ko-Rado theorem~\cite{EKR61} determines
the  maximum independent sets of Kneser graphs.
 Recalling the fact that  an independent set in a graph $G$ is a subset of vertices containing no subgraph isomorphic to $K_2$, the following question was asked in \cite{MA}. 
 
 ``Given a graph $G$, how large a family $\F\subseteq {[n]\choose k}$ must be chosen to guarantee that $\KG_{n,k}[\F]$ has some subgraph isomorphic to $G$? What is the structure of the largest subset $\F\subseteq {[n]\choose k}$ for which   $\KG_{n,k}[\F]$ has no subgraph isomorphic to $G$?"

 This problem has  already been investigated for some special cases.
 In particular, if  $G=K_2$, the answer is the Erd{\H o}s-Ko-Rado theorem and
  if   $G=K_{1,t}$ or $G=K_{s,t}$, the question has been studied in~\cite{MA,MR3022158} and \cite{MA,MR3386026}, respectively. 

 If $G=K_{r+1}$, the question is equivalent to the famous Erd{\H o}s matching conjecture~\cite{Erdos65}  where  it has been studied 
extensively in the literature  (e.g.~\cite{ErdGal59,Frankl2017,LucMie14, BolDayErd76,Erdos65,FranLucMie12,HuaLohSud12}). As  one of the strongest results  in this regard, Frankl~\cite{Frankl13}  has confirmed the Erd{\H o}s matching conjecture for $n\geq (2r+1)k-r$.

Also, this  problem can be  considered  as a vertex Tur{\'a}n problem as follows. Given a host graph $H$ (which is the Kneser graph in our case) 
and a forbidden graph $G$, what is  the size and structure of the largest set $U\subseteq V(H)$ such that the induced subgraph $H[U ]$ is $G$-free? 

In\cite{MA}, Alishahi and the   author determined  the size and  structure of a family $\F$ on $[n]$ with maximum size such that the induced 
subgraph $\KG_{n,k}[\F]$
is $G$-free provided that $n$ is sufficiently large.  

In the sequel, a subgraph $H$ of a given graph $G$ is called a {\it special} subgraph  if removing its vertices from $G$ reduces the chromatic number by one.

 \begin{alphtheorem}\label{thm:MA}{\rm \cite{MA}}
Let $k\geq 2$ be a fixed positive integer and $G$ be a fixed graph  for which $\chi(G)$ is the chromatic number and
 $\eta(G)$ is the minimum possible size of a color class of $G$ over all possible proper
$\chi(G)$-colorings of $G$. 
There exists a threshold $N(G,k)$ such that 
for any $n\geq N(G,k)$ and for any  $\F\subseteq {[n]\choose k}$, if $\KG_{n,k}[\F]$ has no subgraph  isomorphic to $G$, then 
$$|\F|\leq {n\choose k}-{n-\chi(G)+1\choose k}+\eta(G)-1.$$ 
Moreover, equality holds if and only if there is a $(\chi(G)-1)$-set $L\subseteq [n]$ such that 
$$|\F\setminus (\bigcup_{i\in L}\St_i) |= \eta(G)-1$$
 and $\KG_{n,k}[\F\setminus (\bigcup_{i\in L}\St_i)]$ has no subgraph isomorphic to  a special subgraph of $G$. 
\end{alphtheorem}
 
  Let $s$ and $t$ be two positive integers such that $t\geq s$. A family of $k$-sets $\F$ on $[n]$ is said to be an {\it $(s,t)$-union intersecting family}
   if   for any subfamily $\{A_1,A_2,\ldots,A_{s+t}\}$ of $\F$, 
$$\left(\bigcup\limits_{i=1}^s A_i\right)\cap\left(\bigcup\limits_{i=1}^tA_{s+i}\right)\neq \varnothing.$$ 
It is  straight forward to see that a family $\F$ is an $(s,t)$-union intersecting family if and only if  $\KG_{n,k}[\F]$ is $K_{s,t}$-free.
As a generalization of the Erd{\H o}s-Ko-Rado theorem in \cite{MR3386026} Katona and Nagy showed that for sufficiently large $n$, any $(s,t)$-union intersecting family has cardinality at most 
${n-1\choose k-1}+s-1$. Alishahi and the  author improved this result,  and moreover,   characterized  the extremal cases  in~\cite{MA}.
Also, in \cite{MA} an asymptotic   Hilton-Milner-type stability theorem  was proved for an $(s,t)$-union intersecting family on  [n]. 
 More recently, an explicit  extension of this result
 is proved by Grebner, Methuku, Nagy, Patk{\'o}s, and Vizer \cite{GR}. 
They show  that for $2\leq s\leq t$, the size of an $(s,t)$-union intersecting family on [n], which is not isomorphic to a subfamily of 
$$ \St_1\cup\{F_j|\, 1\leq j\leq s-1,\,\,1\not\in F_j\}$$  for some $F_1,\ldots, F_{s-1}$,
 is at most ${n-1\choose k-1}-{n-sk-1\choose k-1}+s+t-1$ and characterize the largest one. In fact,
they prove that a Hilton-Milner-type theorem  for an $(s,t)$-union intersecting family 
is true when $t\geq s\geq2 $ and $n$ is sufficiently large.  

 Note that the first largest $(s,t)$-union intersecting family is the union of  the star $\St_1$ 
and $s-1$ other $k$-sets. 
 For $i\geq 2$, we say  $\F$ is the $i$th largest  $(s,t)$-union intersecting family, if  $\F$  is  
   a maximal $(s,t)$-union intersecting subfamily of $[n]\choose k$  and is not contained in the $j$th largest $(s,t)$-union intersecting family  for  every $j\leq i-1$.  Indeed, the Hilton-Milner theorem  determines the size and structure of the second $(1,1)$-union intersecting family. Also, Han and Kohayakawa in~\cite{HK} characterize the size and structure of the third $(1,1)$-union intersecting family.
  For sufficiently large $n$, Kostochka and Mubayi in~\cite{KoMu} and Kupavskii in \cite{Kup} find the size and structure of the $i$th $(1,1)$-union intersecting family when $i\leq k+1$. In this regard, for sufficiently large $n$, Grebner  et al. in \cite{GR} determine the size and structure of the second largest $(s,t)$-union intersecting family when $t\geq s\geq 2$. Motivated by the mentioned results, one may naturally ask the following questions.
  
  
   For a family $\F$ and an integer $r\geq 2$, let $\ell_r(\F)$ denote the minimum number $m$ such that by removing
 $m$ sets from $\F$, the resulting family has no $r$  pairwise disjoint sets.  For simplicity of notation, let $\ell(\F)\isdef\ell_2(\F)$.
  \begin{question}
      What are the size and   structure of the $i$th largest $(s,t)$-union intersecting family? 
         \end{question}
  \begin{question}
     What are the size and   structure of the  largest $(s,t)$-union intersecting family with $\ell(\F)\geq s+\beta$? 
   \end{question}

It is worth  mentioning that each family $\F$ with  $\ell(\F)=s-1$ is $(s,t)$-union intersecting and the largest  $(s,t)$-union intersecting
 family 
  $$\F\isdef\St_1\cup\{A_i|\, 1\leq i\leq s-1,\, 1\notin A_i\}$$
 has $\ell(\F)=s-1$.  
Grebner  et al. in \cite{GR}, as their main result, determine the size and structure of the  largest $(s,t)$-union intersecting family with $\ell(\F)\geq s$, when $t\geq s\geq 2$ and $n$ is sufficiently large. 
 By using the Hilton-Milner theorem and their result, one can verify that the second largest $(s,t)$-union intersecting family must have 
 $\ell(F)\geq s$. In fact, the next theorem determines the second largest $(s,t)$-union intersecting family. 
 
 \begin{alphtheorem}\label{thm:GR}{\rm\cite{GR}}
For any $2\leq s\leq t$ and $k$ there exists $N=N(s,t,k)$ such that if $n\geq N$ and $\F$ is a family with $\ell(\F)\geq s$ and $\KG_{n,k}[\F]$ is $K_{s,t}$-free, then we have $$|\F|\leq {n-1\choose k-1}-{n-sk-1\choose k-1}+s+t-1.$$ Moreover, equality holds if and only if $\F$ is isomorphic to some $\F_{s,t}$ which  is defined as follows,
$$\F_{s,t}\isdef\{A:1\in A, A\cap[2,sk+1]\neq\varnothing\}\cup\{A_1,\dots,A_{s}\}\cup\{F_1,\dots,F_{t-1}\}$$
where $A_i\isdef[(i-1)k+2,ik+1]$ for each $1\leq i\leq s$,  and for each $j\leq t-1$, we have $1\in F_j$ and $F_j\cap[2,sk+1]=\varnothing$.
\end{alphtheorem}  
  Motivated by the mentioned results and questions,
in this paper,  we try to determine 
  the structure and size of an $(s,t)$-union intersecting family with maximum size when $\ell(\F)\geq s+\beta$ and $n$ is sufficiently large. 

 To state our main results, we  need  the following  definitions.
\begin{definition}\label{def:1}
 Let $n,k, s,$ and $\beta$ be fixed nonnegative integers. 
 Let $A_1,\ldots,A_{s+\beta}$ be $s+\beta$ pairwise distinct $k$-sets on $[n]$ such that $1\not\in\cup_{i=1}^{s+\beta}A_i$.
 Define $\St_1(A_1,\ldots,A_{s+\beta}:s)$  as   the largest subfamily of $\St_1$ such that each $A\in\St_1(A_1,\ldots,A_{s+\beta}:s)$ is disjoint from at most $s-1$ of $A_i$s. Also, define
 $$T(A_1,\ldots,A_{s+\beta}:s)\isdef\{x|\, \text{there exist distinct}\,\,  i_1,i_2,\ldots,i_{\beta+1}\, \text{such that}\, \, x\in \cap_{j=1}^{\beta+1}A_{i_j}\}.$$

 \end{definition}
 Note that when $\beta=0$, we have  $T(A_1,\ldots,A_{s}:s)=\cup_{i=1}^{s}A_i$
and $\St_1(A_1,\ldots,A_s:s)$ is equal to $\St_1\setminus\{A:A\cap(\cup_{i=1}^{s}A_i)=\varnothing\}$. Also,
 when $s=1$  the family  $\St_1(A_1,\ldots,A_{1+\beta}:1)$ is equal to $\St_1\setminus\{A|\,  A\cap A_i=\varnothing\,\, \text{for some }\, 1\leq i\leq \beta+1\}$ and
  $T(A_1,\ldots,A_{1+\beta}:1)=\cap_{i=1}^{1+\beta}A_i$.
  \begin{definition}
 Let $k, s,$ and $\beta$ be fixed nonnegative integers. 
   If $\lfloor{(s+{\beta})k\over {\beta}+1}\rfloor>k$, define  $\hat{\beta}\isdef\hat{\beta}(k,s,\beta)$  as the  largest positive integer such that
 $\lfloor{(s+{\beta})k\over {\beta}+1}\rfloor=\lfloor{(s+\hat{\beta})k\over \hat{\beta}+1}\rfloor$; else 
   if $\lfloor{(s+{\beta})k\over {\beta}+1}\rfloor=k$, define  $\hat{\beta}\isdef\beta$. 
   \end{definition}
 Now, we are in a position to state our first result.
\begin{theorem}\label{HMnew3}
 Let $k\geq 3, t\geq s\geq 1$, and $\beta$ be fixed  nonnegative integers and $n=n(s,t,k,\beta)$ be sufficiently large. 
Let $\F$ be an $(s,t)$-union intersecting family  such that $\ell(\F)\geq s+\beta$.
 Then  $$|\F|\leq {n-1\choose k-1}-{n-\lfloor{(s+{\beta})k\over {\beta}+1}\rfloor-1\choose k-1}+s+t+\hat{\beta}-1.$$
Equality holds if and only if  there exist  pairwise distinct $k$-sets
$A_1,\ldots,A_{s+\hat{\beta}}$ and  $F_1,\ldots, F_{t-1}$
such that 
\begin{enumerate}
\item $1\notin\bigcup\limits_{i=1}^{s+\hat{\beta}} A_i$,
\item\label{ii} $|T(A_1, \ldots,A_{s+\hat{\beta}}:s)|=\lfloor{(s+\beta)k\over \beta+1}\rfloor$, 
\item for each $i\leq t-1$, $F_i\in\St_1\setminus\St_1(A_1, \ldots,A_{s+\hat{\beta}}:s),$ and
\item\label{iv} the family $\{A_1, \ldots,A_{s+\hat{\beta}},F_1,\ldots, F_{t-1}\}$ is an $(s,t)$-union intersecting family
\end{enumerate}
and $\F$ is  isomorphic to $\St_1(A_1,\ldots,A_{s+\hat{\beta}}:s)\cup\{A_1,\ldots,A_{s+\hat{\beta}}\}\cup\{F_1,\ldots, F_{t-1}\}.$
\end{theorem}
It is worth mentioning that  Theorem~\ref{thm:GR} follows from Theorem~\ref{HMnew3} by choosing   $\beta=0$ and $s\geq 2$.
By applying the previous theorem and using  some properties of $T(A_1, \ldots,A_{s+{\beta}}:s)$,
we can find out  the $j$th largest $(s,t)$-union intersecting family for some $j$'s.
We provide a more detailed analysis in our remarks proceeding 
 the proof of Theorem~\ref{HMnew3}.

 Note that perhaps for some $k,s,$ and $\beta$ there exist no distinct pairwise $A_1, \ldots,A_{s+\beta}$ satisfying Condition~(\ref{ii}) in the previous theorem. For example,  one may choose $k=3, s=3$, and $\beta=5$. Thus,
we have $\lfloor{(s+\beta)k\over \beta+1}\rfloor=4$. Since $\cup_{i=1}^{8}A_i=T(A_1, \ldots,A_8:3)$, if there exist  $A_1, \ldots,A_8$ for which 
$|T(A_1, \ldots,A_8:3)|=4$, then at least  two of $A_i$'s must be identical, which is not possible.
Therefore, for some $k,s,$ and $\beta$ there do not exist  any  $A_1,  \ldots,A_{s+\beta}$ such that 
 $|T(A_1, \ldots,A_{s+{\beta}}:s)|= \lfloor{(s+\beta)k\over \beta+1}\rfloor$. Consequently, as we will show in the proof of Theorem~\ref{HMnew3},
 each $(s,t)$-union intersecting  family $\F$ is of size less than
 ${n-1\choose k-1}-{n-\lfloor{(s+{\beta})k\over {\beta}+1}\rfloor-1\choose k-1}$ showing that 
 if  $\ell(\F)\geq s+\beta$, then $|\F|$ is at most
 ${n-1\choose k-1}-{n-|T(A_1, \ldots,A_{s+{\beta}}:s)|-1\choose k-1}+O(n^{k-3})$.  

 When  $\F$ is a $(1,t)$-union intersecting family of ${[n]\choose k}$ (or $\KG_{n,k}[\F]$ is a $K_{1,t}$-free subgraph of $\KG_{n,k}$) 
 it is proved that  every $(1,t)$-union intersecting family
with at least ${n-1\choose k-1}-{n-k-1\choose k-1}+(t-1){2k-1\choose k-1}+t$ members is contained 
in  some star $\St_i$ for sufficiently large $n$~\cite{MA}.
Moreover, it  is posed as a conjecture in the same refrence that 
for sufficiently large $n$ one can replace the term  $(t-1){2k-1\choose k-1}$ by $1$. Note that this conjecture is an extension of  the Hilton-Milner  theorem.
Also, if the statement of Theorem~\ref{thm:GR} is true for $s=1$, then the conjecture holds, however, the  condition $s\ge 2$ is necessary in the proof of  Theorem~\ref{thm:GR} presented in~\cite{GR}.  This conjecture is one of our  motivations for this study, in which we show that the conjecture follows from Theorem~\ref{HMnew3} choosing $s=1$
and $\beta=0$. For further reference we state this fact in the following corollary.

\begin{corollary}\label{cor1}
Let $k\geq 3$ and $t$ be positive integers  and $n=n(k,t)$ is sufficiently  large. 
Any $(1,t)$-union intersecting family $\F\subseteq {[n]\choose k}$, which is not contained in any star, has  cardinality  at most $${n-1\choose k-1}-{n-k-1\choose k-1}+t.$$
Equality holds if and only if $\F$ is isomorphic to
$\J_{1}^{1,t}\isdef\St_1(A_1:1)\cup \{A_1\}\cup\{F_1,\ldots, F_{t-1}\}$ where  $F_i\in\St_1\setminus\St_1(A_1:1)$ for each $i\leq t-1$.
\end{corollary}
Note that if $F\in\St_1\setminus\St_1(A_1:1)$, it means that $1\in F$ and $F\cap A_1=\varnothing$. Also, note that $\J_1^{1,1}$ is isomorphic to $\HM$ and $\J_1$.

Concerning our next result when $s=1, t\geq 1$, and $\beta\leq k-3$,
 motivated by Theorems~\ref{HK} and \ref{Kup} and the mentioned conjecture, 
  we  determine the maximum size and structure of  a $(1,t)$-union intersecting family
$\F$ with $\ell(\F)\geq 1+\beta$. Note that when $s=1$ and $\beta\geq 1$, Theorem~\ref{HMnew3} does not give a sharp bound for maximum size of $(1,t)$-union intersecting families. This result leads us to determine  the $i$th largest  $(1,t)$-union intersecting families
 where $i\leq k-2$.

Before stating the next result we need to  introduce the following construction.
\begin{definition}\label{def2}
Let $i\leq k-1$ be a nonnegative integer. For any $(i+1)$-set $J=\{1,x_1,\ldots,x_i\}$ of  $[n]$	 and any $(k-1)$-set $E\subset [n]\setminus J$.
 Let $A_1,\ldots,A_i$ be $i$ $k$-subsets on $[n]\setminus\{1\}$ such that
 $\cap_{j=1}^i A_j=E$ and $A_j\setminus E=\{x_j\}$ for each $j\leq i$ define $\J^{1,t}_{i}$ as follows
 $$\J^{1,t}_{i}\isdef \St_1(A_1,\ldots,A_i:1)\cup\{A_1,\ldots,A_i\}
 \cup\B_1\cup\cdots\cup \B_i,$$
where   $\B_j$, for $j\leq i$,  defined as follows
 $$\B_j\isdef\{B_{p}: p\in[t-1], B_p\cap E=\varnothing,\, J\setminus B_p=\{x_j\}\}$$
\end{definition}
 Notice that $\J^{1,1}_{i}$    isomorphic to $\J_i$ and 
 $\J_{i}=\St_1(A_1,\ldots,A_i:1)\cup\{A_1,\ldots,A_i\}$.
Since  $\B_j$'s in the definition of $\J^{1,t}_{i}$ are pairwise disjoint. Therefore, $|\J^{1,t}_{i}|=|\J_{i}|+i(t-1)$.  

For $s=1$ we can state a strong improvement of Corollary~\ref{cor1} and Theorem~\ref{HMnew3} as follows.
 \begin{theorem}\label{HMnew2}
 Let $k\geq 5, t\geq 1$, and $\gamma=1+\beta\leq k-2$ be nonnegative integers and  $n=n(k,t,\gamma)$ is sufficiently large. 
Let $\F$ be a $(1,t)$-union intersecting family with $\ell(\F)\geq \gamma$. Then 
$$|\F|\leq {n-1\choose k-1}-{n-k\choose k-1}+{n-k-\gamma\choose k-\gamma-1}+\gamma t.$$
Equality holds if and only if $\F$ is  isomorphic to 
 $\J_{\gamma}^{1,t}$.
  \end{theorem}
  It can be seen that the next corollary is a direct consequence of  Theorem~\ref{HMnew2}.
   Notice that  we   need to apply  Theorem~\ref{Kup} to prove it.
    
\begin{corollary}
Let $n$, $k\geq 5$,  $t\geq1$, and $\gamma\leq k-2$  be nonnegative integers such that  $n=n(k,t,\gamma)$ is sufficiently large. 
Let $\F$ be a $(1,t)$-union intersecting family  that is not isomorphic to a subfamily of    $\J_i\cup\B$ where 
$\B\subseteq \St_1\setminus\J_i$ and  $0\leq i\leq \gamma-1$.
Then $$|\F|\leq {n-1\choose k-1}-{n-k\choose k-1}+{n-k-\gamma\choose k-\gamma-1}+\gamma t.$$
Equality holds if and only if $\F$ is isomorphic to some $\J_{\gamma}^{1,t}$.
\end{corollary}
Note that if we choose $\B=\B_1\cup\ldots\cup\B_i$, where $\B_j$'s come from Definition~\ref{def2}, then $\J_{i}\cup\B=\J_{i}^{1,t}$.
\subsection{ Some stability results for the Erd{\H o}s matching conjecture and its generalization} 
 The  Erd{\H o}s matching conjecture is one of the famous open problems in extremal set theory. It states that for $n\geq (r+1)k$,
 the size of the largest subset $\F\subseteq {[n]\choose k}$ for which $\KG_{n,k}[\F]$ has no copy of $K_{r+1}$ is 
 $\max\{{(r+1)k-1\choose k},{n\choose k}-{n-r\choose k}\}$. 
In recent years, this conjecture has received considerable attention. It has been already proved that the conjecture is true for $k\leq 3$  (see~\cite{ErdGal59,Frankl2017,LucMie14}). Also, improving the earlier results
 in~\cite{BolDayErd76,Erdos65,FranLucMie12,HuaLohSud12}, Frankl~\cite{Frankl13} confirmed the conjecture for $n\geq (2r+1)k-r$; moreover, he determined the structure of  the extremal  cases in this range. Frankl and Kupavskii~\cite{FK} proved a Hilton-Milner-type stability theorem for the Erd{\H o}s matching conjecture for $n\geq  (2+o_{r}(1))(r+1)k $ as 
 a significant improvement of a classical result due to Bollob{\' a}s, Daykin and Erd{\H o}s~\cite{BolDayErd76}.

Hereafter, we will focus on complete multipartite  graphs $K_{s_1, s_2,\cdots, s_{r+1}}$ as a forbidden subgraph. We show that  
the previous results for $(s,t)$-union intersecting family  can be extended to $K_{s_1, s_2,\cdots, s_{r+1}}$-free subgraph of Kneser graphs instead of  $K_{s,t}$-free subgraphs of Kneser graphs as  nontrivial extensions of the Erd{\H o}s matching conjecture. 
 In this regard,
Grebner et. al. show that a generalization of Theorem~\ref{thm:GR} holds when $\KG_{n,k}[\F]$ is $K_{s_1, s_2,\cdots, s_{r+1}}$-free 
when $s_1\geq \cdots\geq s_{r+1}\geq 2$. They determine the size and structure of the second largest family $\F$ on $[n]$ such that $\KG_{n,k}[\F]$
 is $K_{s_1,s_2,\ldots,s_{r+1}}$-free,  where $s_{r+1}\geq 2$ for sufficiently large $n$. 
 Before stating their result, we need an extension of the construction of Definition~\ref{def:1}.

 \begin{definition}\label{dfs}
 Let $n,k, s,$ and $\beta$ be fixed positive integers. Let $A_1,\ldots,A_{s+\beta}$ be $s+\beta$ pairwise distinct $k$-sets on $[n]$ such that $[r]\bigcap(\cup_{i=1}^{s+\beta}A_i)=\varnothing$.
 Define $\St_r^{[r-1]}(A_1,\ldots,A_{s+\beta}:s)$ as the largest subfamily of $\St_r$ such that 
 each $A\in\St_r^{[r-1]}(A_1,\ldots,A_{s+\beta}:s)$ is disjoint from $[r-1]$ and at most $s-1$ of $A_i$'s.
 \end{definition}
Note that the family $\St_r(A_1,\ldots,A_{s+\beta}:s)$  in Definition~\ref{def:1} is a special case of  Definition~\ref{dfs} when $r=1$.

 \begin{alphtheorem}\label{thm:GR2}{\rm\cite{GR}}
 For any $k\geq 2$ and integers
 $  s_1\geq s_2\cdots\geq s_{r+1}\geq 2$ there exists $N=N(s_1, s_2,\dots, s_{r+1},k)$ such that if $n\geq N$ and $\F$ is a family with $\ell_{r+1}(F)\geq s_{r+1}$ 
 and $\KG_{n,k}[F]$ is $K_{s_1, s_2,\cdots, s_{r+1}}$-free, then we have 
 $$|\F|\leq {n\choose k}-{n-r\choose k}-{n-s_{r+1}k-r\choose k-1}+s_{r}+s_{r+1}-1.$$
Moreover, equality holds if and only if $\F$ is isomorphic to $\St_r^{[r-1]}(A_1,\ldots,A_s:s)\cup\{A_1,\ldots,A_s\}\cup\{F_1,\ldots,F_{t-1}\}$.
\end{alphtheorem}
We are able to prove an analog of the previous theorem by using the Erd{\H o}s-Stone-Simonovits theorem and Theorem~\ref{HMnew3}.
\begin{theorem}\label{last}
 Let $k\geq 3, s_1\geq\cdots \geq s_{r+1}\geq 1$ and $\beta$ be fixed nonnegative integers 
and $n=n(s_1,\ldots,s_{r+1},k,\beta)$ is sufficiently large. 
 Assume that  $\hat{\beta}=\hat{\beta}(k,s_{r+1},\beta)$.
 Let $\KG_{n,k}[\F]$ be  $K_{(s_1,\ldots,s_{r+1})}$-free  such that $\ell_{r+1}(\F)\geq s_{r+1}+\beta$.
 Then  
 $$|\F|\leq {n\choose k}-{n-r\choose k}-{n-\lfloor{(s_{r+1}+{\beta})k\over {\beta}+1}\rfloor-r\choose k-1}+s_{r}+s_{r+1}+\hat{\beta}-1.$$
Equality holds if and only if  there exist $s_{r+1}+\hat{\beta}$ pairwise distinct $k$-sets
$A_1,\ldots,A_{s_{r+1}+\hat{\beta}}$
such that 
 \begin{itemize}
\item $[r]\bigcap(\bigcup\limits_{i=1}^{s_{r+1}+\hat{\beta}}A_i)=\varnothing$, 
\item $|T(A_1, \ldots,A_{s_{r+1}+\hat{\beta}})|=\lfloor{(s_{r+1}+\beta)k\over \beta+1}\rfloor$, 
\item  for each $i\leq s_r-1$,
$F_i\in\St_r\setminus\St_r^{[r-1]}(A_1, \ldots,A_{s_{r+1}+\hat{\beta}}:s_{r+1})$ and $F_i\cap [r-1]=\varnothing$, and
 \item the family $\{A_1, \ldots,A_{s_{r+1}+\hat{\beta}},F_1,\ldots, F_{s_r-1}\}$ is an $(s_{r+1},s_r)$-union intersecting family and
 \end{itemize}
$\F$ is  isomorphic to 
$$\bigcup_{i=1}^{r-1}\St_i\cup\St_r^{[r-1]}(A_1, \ldots,A_{s_{r+1}+\hat{\beta}}:s_{r+1})\cup\{A_1, \ldots,A_{s_{r+1}+\hat{\beta}}\}\cup\{F_1,\ldots, F_{s_r-1}\}$$

\end{theorem}

When $s_{r+1}=1$ same as Theorem~\ref{HMnew2} we are able to prove
 a stronger  result than Theorem~\ref{last} which yields a new stability result
for Erd{\H o}s matching conjecture for sufficiently large $n$. 
    \begin{definition}\label{def3}
Let $i\leq k-1$ be a nonnegative integer. For any $(i+r)$-set $J=\{1,\ldots,r,x_1,\ldots,x_i\}$ of  $[n]$	
 and any $(k-1)$-set $E\subset [n]\setminus J$. Let $A_1,\ldots,A_i$ be $i$ $k$-subsets on $[n]\setminus[r]$ such that
 $\cap_{j=1}^i A_j=E$ and $A_j\setminus E=\{x_j\}$ for each $j\leq i$  define $\J^{1,t}_{i,r}$ and $\J'^{1,t}_{i,r}$ as follows
 $$\J^{1,t}_{i,r}\isdef \St_r^{[r-1]}(A_1,\ldots,A_i:1)
 \cup\B_1\cup\cdots\cup \B_i$$
where   $\B_j$, for $j\leq i$    defined as follows,
 $$\B_j\isdef\{B_{p}: p\in[t-1], B_p\cap E=\varnothing,\, J\setminus B_p=\{1,\ldots,r-1,x_j\}\}.$$
\end{definition}
 Notice that $\J^{1,t}_{i,1}$ is isomorphic to $\J^{1,t}_{i}$. 
Now we are in a position to state a stability result related to Erd{\H o}s matching conjecture provided that $n$ is sufficiently large.
\begin{theorem}\label{last1}
 Let $k\geq 5, s_1\geq\cdots \geq s_{r}\geq 1$, and $\gamma(=1+\beta)\leq k-2$ be fixed nonnegative integers  such that $n=n(s_1,\ldots,s_{r},k,\gamma)$ is sufficiently large.
Let $\KG_{n,k}[\F]$ be  $K_{(s_1,\ldots,s_{r},1)}$-free  such that $\ell_{r+1}(\F)\geq \gamma$.
 Then  
 $$|\F|\leq {n\choose k}-{n-r\choose k}-{n-k-r+1\choose k-1}+{n-k-r-\gamma+1\choose k-\gamma-1}+\gamma t.$$

Equality holds if and only if $\F$ is  isomorphic to $\J^{1,s_r}_{\gamma,r}$
\end{theorem}
  \begin{corollary}\label{lastcor}
Let $n$, $k\geq 5$,  $s_1\geq\cdots \geq s_{r}\geq 1$, and $\gamma(=1+\beta)\leq k-2$ be nonnegative integers such that   $n=n(s_1,\ldots,s_{r},k,\gamma)$  is sufficiently large. 
Let $\F$ be a family such that  $\KG_{n,k}[\F]$ is  $K_{(s_1,\ldots,s_{r},1)}$-free  and $\F$ 
 is not isomorphic to a subfamily of    $\J^{1,1}_{i,r}\cup\B$ where 
$\B\subseteq\St_r\setminus\J^{1,1}_{i,r}$ and  $0\leq i\leq \gamma-1$.
Then 
 $$|\F|\leq {n\choose k}-{n-r\choose k}-{n-k-r+1\choose k-1}+{n-k-r-\gamma+1\choose k-\gamma-1}+\gamma t.$$
Equality holds if and only if $\F$ is  isomorphic to $\J^{1,s_r}_{\gamma,r}$
\end{corollary}

\section{ Proofs}

Before  the proof of Theorem~\ref{HMnew3}, let us state an interesting lemma from \cite{GR}.
 Here we show that a strong generalization of Lemma~\ref{GR:lem} is true.

\begin{alphlemma}{\rm\cite{GR}}\label{GR:lem}
Let $s\leq t$ and let $A_1,A_2,\ldots, A_{s+1}$ be $k$-sets on $[n]$ such that $1\not\in\cup_{i=1}^{s+1}A_i$.
 Suppose that  $\F'$ is a subfamily of $\St_1$ such that for 
$\F= \F'\cup\{A_1,A_2,\ldots, A_{s+1}\}$ the induced subgraph of  $\KG_{n,k}[\F]$ is $K_{s,t}$-free. 
There exists $n_0=n(k, s, t)$ such that if $n\geq n_0$ holds, then we have
$$|\F|\leq {n-1\choose k-1}-{n-\lfloor{(s+1)k\over 2}\rfloor-1\choose k-1}+(s+1)t.$$
\end{alphlemma}
 The next lemma provides an interesting and useful generalization of Lemma~\ref{GR:lem}. I believe that Lemma~\ref{B:lem} independently will be a useful result and will have more applications.
\begin{lemma}\label{B:lem}
Let $k, s,$ and $\beta$ be fixed nonnegative integers and $n=n(k,s,\beta)$ is sufficiently large.  Let $A_1,A_2,\ldots, A_{s+\beta}$ be $k$-sets on $[n]$ such that $1\not\in\cup_{i=1}^{s+\beta}A_i$. Then
\begin{enumerate}
\item[(a)] 
$ {n-1\choose k-1}-{n-|T(A_1,\ldots,A_{s+\beta}:s)|-1\choose k-1}\leq|\St_1(A_1,\ldots,A_{s+\beta}:s)|.$
 \item[(b)] 
$|\St_1(A_1,\ldots,A_{s+\beta}:s)|\leq {n-1\choose k-1}-{n-\lfloor{(s+\beta)k\over \beta+1}\rfloor-1\choose k-1}$
and  equality holds if and only if $$|T(A_1,\ldots,A_{s+\beta}:s)|=\lfloor{(s+\beta)k\over \beta+1}\rfloor.$$
In particular,  if 
$|T(A_1,\ldots,A_{s+\beta}:s)|<\lfloor{(s+\beta)k\over \beta+1}\rfloor$, then  
$$ |\St_1(A_1,\ldots,A_{s+\beta}:s)|= {n-1\choose k-1}-{n-|T(A_1,\ldots,A_{s+\beta}:s)|-1\choose k-1}+O(n^{k-3}).$$ 
\item[(c)]
For $s=1$, we have
$|\St_1(A_1,\ldots,A_{1+\beta}:1)|\leq 
{n-1\choose k-1}-{n-k\choose k-1}+ {n-k-\beta-1\choose k-\beta-2}.$
Moreover for $\beta\geq1$, equality holds if and only if $|T(A_1,\ldots,A_{1+\beta}:1)|=k-1.$

\noindent In particular,  if 
$|T(A_1,\ldots,A_{1+\beta}:1)|< k-1$, then  
$$ |\St_1(A_1,\ldots,A_{1+\beta}:1)|= {n-1\choose k-1}-{n-|T(A_1,\ldots,A_{1+\beta}:1)|-1\choose k-1}+O(n^{k-3}).$$ 

\end{enumerate}
\end{lemma}
\begin{proof}
For abbreviation let $T(A_1,\ldots,A_{s+\beta}:s)=T_{\beta}$. 
For the proof of Part (a), let   $1\in A$. If $A\cap T_{\beta}\neq\varnothing$, then
 $A$ is disjoint from at most $s-1$ sets of $A_1,A_2,\ldots, A_{s+\beta}$. Therefore,
 $ {n-1\choose k-1}-{n-|T_{\beta}|-1\choose k-1}\leq |\St_1({A_1,\ldots,A_{s+\beta}}:s)|.$

Now we prove Part (b). One can check that $|T_\beta|\leq \lfloor{(s+\beta)k\over \beta+1}\rfloor.$
Assume that $|T_{\beta}|<\lfloor{(s+\beta)k\over\beta+1}\rfloor$.  Let $A\in\St_1(A_1,\ldots,A_{s+\beta}:s)$. Therefore,
$A$ intersects  at least $\beta+1$ of $A_1,\ldots,A_{s+\beta}$. We have two possibilities for $A$. Either $A\cap T_\beta\neq\varnothing$ or
$A\cap T_\beta=\varnothing$ and $A$ intersects  at least $\beta+1$ of $A_1,\ldots,A_{s+\beta}$.
The number of members in $\St_1$ which meet $T_{\beta}$
is equal to  ${n-1\choose k-1}-{n-|T_{\beta}|-1\choose k-1}$.
The number of $k$-sets in $\St_1$, which intersect  at least $\beta+1$ of $A_1,\ldots,A_{s+\beta}$ and have no common
element with $T_\beta$, is at most 
\begin{align*}
& \sum_{i_1,\ldots,i_{\beta+1}\in[s+\beta]}\,\,\sum_{\varnothing\neq  B_1\subseteq A_{i_1}\setminus T_{\beta}}\cdots\sum_{ \varnothing\neq B_{\beta+1}\subseteq A_{i_{\beta+1}}\setminus T_{\beta}}
{n-| T_{\beta}|-|\cup_{i=1}^{\beta+1} B_i|-1\choose k-|\cup_{i=1}^{\beta+1}B_i|-1}\\
&\leq \sum_{i_1,\ldots,i_{\beta+1}\in[s+\beta]}\,\,\prod_{j=1}^{\beta+1}2^{|A_{i_j}\setminus T_{\beta}|}{n-|T_{\beta}|-3\choose k-3}\\
&\leq {s+\beta \choose 1+\beta}2^{k(\beta+1)}{n-|T_{\beta}|-3\choose k-3}.
\end{align*}
Therefore,
\begin{align}\label{eq2}
|\St_1(A_1,\ldots,A_{s+\beta}:s)|&\leq {n-1\choose k-1}-{n-|T_{\beta}|-1\choose k-1}\\
 ~&+\sum_{i_1,\ldots,i_{\beta+1}\in[s+\beta]}\sum\limits_{\substack{\varnothing\neq B_1\subseteq A_{i_1}\setminus T_{\beta} \\ 
 \vdots\\{	\varnothing\neq B_{\beta+1}\subseteq A_{i_{\beta+1}}\setminus T_{\beta}}}}
{n-| T_{\beta}|-|\cup_{i=1}^{\beta+1}B_i|-1\choose k-|\cup_{i=1}^{\beta+1}B_i|-1} \nonumber\\
~&\leq {n-1\choose k-1}-{n-|T_{\beta}|-1\choose k-1} + {s+\beta \choose 1+\beta}2^{k(\beta+1)}{n-|T_{\beta}|-3\choose k-3} \nonumber
\end{align}
 and then
\begin{align*}
|\St_1(A_1,\ldots,A_{s+\beta}:s)|&\leq {n-1\choose k-1}-{n-|T_{\beta}|-1\choose k-1} + O(n^{k-3})\\
&={n-1\choose k-1}-{n-\lfloor{(s+\beta)k\over \beta+1}\rfloor-1\choose k-1}-
\sum_{i=|T_\beta|+1}^{\lfloor{(s+\beta)k\over \beta+1}\rfloor} {n-i-1\choose k-2} + O(n^{k-3})\\
&<{n-1\choose k-1}-{n-\lfloor{(s+\beta)k\over \beta+1}\rfloor-1\choose k-1}
\end{align*}
provided that $n$ is sufficiently large.

Now assume that we have the equality $|\St_1(A_1,\ldots,A_{s+\beta}:s)|={n-1\choose k-1}-{n-\lfloor{(s+\beta)k\over \beta+1}\rfloor-1\choose k-1}.$ 
By contradiction assume that $|T_\beta|< \lfloor{(s+\beta)k\over \beta+1}\rfloor$.  Using the same reasoning one may verify that  
 when $n$ is sufficiently large, then $|\St_1(A_1,\ldots,A_{s+\beta}:s)|$ is less than ${n-1\choose k-1}-{n-\lfloor{(s+\beta)k\over \beta+1}\rfloor-1\choose k-1}$ which is not possible.

Now suppose that  $|T_{\beta}|=\lfloor{(s+\beta)k\over \beta+1}\rfloor$. 
To prove the last part of  (b), it suffices  to show that
$$\St_1\setminus\St_1(A_1,\ldots,A_{s+\beta}:s)=\{A|1\in A,\,\, A\cap T_{\beta}=\varnothing\}.$$
From the division algorithm, we know that $(s+\beta)k=\lfloor{(s+\beta)k\over \beta+1}\rfloor(\beta+1)+r$ where $0\leq r\leq \beta$.
Since $|T_{\beta}|=\lfloor{(s+\beta)k\over \beta+1}\rfloor$, there are at most $r\leq\beta$ elements  
  in $\cup_{i\in[s+\beta]}A_i$ which are not in $T_{\beta}$. Therefore, there exist $1\leq i_1<\ldots<i_s\leq s+\beta$ such that 
$A_{i_1}\cup\cdots\cup A_{i_s}\subseteq T_{\beta}.$ On the other hand,  for every $1\leq j_1<\ldots<j_s\leq s+\beta$, we have
 $  T_{\beta}\subseteq A_{j_1}\cup\cdots\cup A_{j_s}.$ Therefore, $ T_{\beta}=A_{i_1}\cup\cdots\cup A_{i_s}.$ 
Assume that $1\in A$ and $A\cap T_{\beta}=\varnothing$. Hence, $A\cap (A_{i_1}\cup A_{i_2}\cup\cdots\cup A_{i_s})=\varnothing$. Therefore,
$A$ is disjoint from at least $s$ $k$-subsets of $A_1,A_2,\ldots,A_{s+\beta}$ and consequently 
$A\in\St_1\setminus\St_1(A_1,\ldots,A_{s+\beta}:s)$.  If $1\in A$ and $A$ is disjoint from at least 
$s$ $k$-subsets of $A_1,A_2,\ldots,A_{s+\beta}$, then it is clear that each element of $A$ appears in at most $\beta$ of $A_i$'s and hence we have 
$A\cap T_{\beta}=\varnothing.$ 

 For the proof of (c), if $|T_{\beta}|\leq k-2$, then the proof is the same as the first part of (b). 
  Hence, we may assume that $|T_{\beta}|$ is $k-1$ or $k$.   Note that when $s=1$,
 $T_{\beta}=\cap_{i=1}^{1+\beta}A_i$.
 If $|T_{\beta}|=k$, then $\beta$ must be equal to $0$ and consequently $|\St_1(A_1:1)|={n-1\choose k-1}-{n-k-1\choose k-1}$. Thus, we may assume that $|T_{\beta}|=|\cap_{i=1}^{1+\beta}A_i|=k-1$ and $\beta\geq 1$. Then, there exist $\beta+1$ 
 elements in $[n]$, say $x_1,\ldots,x_{\beta+1}$, such that 
 $A_j\setminus T_{\beta}= \{x_j\}.$  
 Let $A\in\St_1(A_1,\ldots,A_{1+\beta}:1)$. Therefore,
$A$ intersects each of $A_1,\ldots,A_{1+\beta}$. We have two possibilities for $A$. Either $A\cap T_\beta\neq\varnothing$ or
$A\cap T_\beta=\varnothing$ and $A$ intersects  all of $A_1,\ldots,A_{1+\beta}$.
There are  ${n-1\choose k-1}-{n-k\choose k-1}$ members in $\St_1$ such that $A\cap T_\beta\neq\varnothing$. 
The number of $k$-sets in $\St_1$, which intersect  all of $A_1,\ldots,A_{1+\beta}$ and have no common
element with $T_\beta$, is equal ${n-k-\beta-1\choose k-\beta-2}$. Therefore,
$$|\St_1(A_1,\ldots,A_{1+\beta}:1)|= {n-1\choose k-1}-{n-|T_{\beta}|-1\choose k-1} +{n-k-\beta-1\choose k-\beta-2}.$$
Note that when $\beta\geq k-1$, we have ${n-k-\beta-1\choose k-\beta-2}=0$.
\end{proof}
In the proof of Theorem~\ref{HMnew3}  in addition to Lemma~\ref{B:lem},  we will use  the following  two results. The first one is a classical result on the number of edges of a $K_{s,t}$-free graph and the second one is a result on the number of disjoint pairs in
a family of $k$-sets $\F$. 
  \begin{alphtheorem}{\rm \cite{Kovari1954}}\label{bip}
 For any two positive integers $s\leq t$, if $G$ is a $K_{s,t}$-free graph with $n$ vertices, then the number of edges of $G$ is at most
 $({1\over 2}+o(1))(t-1)^{1\over s} n^{2-{1\over s}}$.
    \end{alphtheorem}
\begin{alphlemma}{\rm\cite{MR3482268}}\label{lem:1}
Let  $\F$ be a family $k$-sets on $[n]$. Then the number  of disjoint pairs in $\F$ is at least 
$ {\ell(\F)^2\over 2{2k\choose k}}.$
\end{alphlemma}
\noindent For an intersecting family $\F'$ on $[n]$, define $\Delta(\F')=\max_{i\in[n]}|\F'\cap \St_i|$.
\begin{proof}[Proof of Theorem~{\rm\ref{HMnew3}}]
Let $\F$ be an $(s,t)$-union intersecting family  of ${[n]\choose k}$ with $\ell(\F)\geq s+\beta$ 
and cardinality 
$$M= {n-1\choose k-1}-{n-\lfloor{(s+\beta)k\over \beta+1}\rfloor-1\choose k-1}+t-1+s+\hat{\beta}.$$
We consider the following three cases.
\begin{enumerate}
\item[(i)] {\it $\ell(\F)= s+\beta'$ where $\beta\leq \beta'\leq \hat{\beta}$.}

\noindent This  implies that  there exist $A_1, A_2,\ldots,A_{s+\beta'}$ in $\F$ such that 
$\F'=\F\setminus\{A_1, A_2,\ldots,A_{s+\beta'}\}$  is an intersecting family. Without loss of generality assume that
$\Delta(\F')$  has  the maximum possible value.
$|\F'|$ is equal to
  $${n-1\choose k-1}-{n-\lfloor{(s+\beta)k\over \beta+1}\rfloor-1\choose k-1}+t-1+\hat{\beta}-\beta'.$$ 
First we show that for each $i\leq s+\beta'$,  $1\not\in A_i$.
If $\F'\subseteq \St_1$, then by the minimality of $\ell(\F)$, each 
 $A_i$ must be  disjoint from at least one member of $\F'\subseteq \St_1$, so $1\not\in \cup_{i=1}^{s+{\beta'}}A_i$.
 If $\F'\not\subseteq \St_1$,
then by the Hilton-Milner theorem,  we conclude that $|\F'|= {n-1\choose k-1}-{n- k-1\choose k-1}+1$. Consequently, 
there exists a unique $B\in\F'$ such that $\F'\setminus \{B\}\subseteq \St_1$ and moreover, we must have $t=2$, $\lfloor{(s+\beta')k\over \beta'+1}\rfloor=k$ and $\beta'=\hat{\beta}$.
 If there is  $A_i$ such that $1\in A_i$, by the  minimality of $\ell(\F)$,
$A_i$ must be disjoint from $B$. Define $\F''= (\F'\setminus \{B\})\cup\{A_i\}$. Hence, $|\F'|=|\F''|$ and $\Delta(\F'')=\Delta(\F')+1$ which contradicts with  the fact that  $\Delta(\F')$  has the maximum possible value. Then, $1\not\in \cup_{i=1}^{s+{\beta'}}A_i$. We now consider the following three subcases.

\begin{enumerate}
\item{\it $\F'\subseteq \St_1$ and $|T(A_1, A_2,\ldots,A_{s+\beta'}:s)|= \lfloor{(s+\beta')k\over \beta'+1}\rfloor$.}
 
\noindent Since $\beta\leq \beta'\leq \hat{\beta}$, we have $\lfloor{(s+\beta')k\over \beta'+1}\rfloor=\lfloor{(s+\beta)k\over \beta+1}\rfloor$.
In view of the last part of the proof of Lemma~\ref{B:lem}~(b), there  are $1\leq i_1<\ldots<i_s\leq s+\beta'$ such that 
$T(A_1,\ldots,A_{s+\beta'}:s)=A_{i_1}\cup\cdots\cup A_{i_s}.$
Also, note that for every $1\leq j_1<\ldots<j_s\leq s+\beta'$,
 we have 
$$A_{i_1}\cup\cdots\cup A_{i_s}=T(A_1,\ldots,A_{s+\beta'}:s)\subseteq A_{j_1}\cup \cdots\cup A_{j_s}.$$
  From this fact and since $\F$ is an $(s,t)$-union intersecting family, 
  the number of elements of $\F'$ which can be disjoint from   $\cup_{\ell=1}^s A_{j_\ell}$  for some $s$ $k$-sets $A_{j_1}, \ldots, A_{j_s}$  of $A_i$'s
  is at most $t-1$, say $F_1,\ldots, F_{t-1}$. Therefore,
  $\F'\subseteq \St_1(A_1,\ldots,A_{s+\beta'}:s)\cup\{F_1,\ldots, F_{t-1}\}$.
   Thus,  by applying Lemma~\ref{B:lem}~(b), we obtain 
$$|\F'|\leq {n-1\choose k-1}-{n-\lfloor{(s+\beta')k\over \beta'+1}\rfloor-1\choose k-1}+t-1$$
and consequently $\beta'=\hat{\beta}$. Therefore,
$$|\F|\leq {n-1\choose k-1}-{n-\lfloor{(s+\beta)k\over \beta+1}\rfloor-1\choose k-1}+t-1+s+\beta'$$

and  equality holds if and only if $\F$ is isomorphic to 
$$\St_1(A_1,\ldots,A_{s+\beta'}:s)\cup\{A_1,\ldots,A_{s+\beta'}\}\cup\{F_1,\ldots, F_{t-1}\}$$ such that 
$|T(A_1,\ldots,A_{s+\beta'}:s)|=\lfloor{(s+\beta')k\over \beta'+1}\rfloor$, $F_i\in\St_1\setminus\St_1(A_1,\ldots,A_{s+\beta'}:s)$, and
the family $\{A_1,\ldots,A_{s+\beta'},F_1,\ldots, F_{t-1}\}$ is an $(s,t)$-union intersecting family.
\item  {\it $\F'\not\subseteq \St_1$ and $|T(A_1, A_2,\ldots,A_{s+\beta'}:s)|= \lfloor{(s+\beta')k\over \beta'+1}\rfloor$.}

\noindent As $\F'\not\subseteq \St_1$,  
 there exist a $k$-set $B\in\F'$ such that $\F'\setminus \{B\}\subseteq \St_1$ and we  have $t=2$, $\lfloor{(s+\beta')k\over \beta'+1}\rfloor=k$, and
$ \beta'=\hat{\beta}$.
Since $|T(A_1, A_2,\ldots,A_{s+\beta'}:s)|= \lfloor{(s+\beta')k\over \beta'+1}\rfloor=k$, 
in view of  
the last part of the proof of Lemma~\ref{B:lem}~(b), there  are $1\leq i_1<\ldots<i_s\leq s+\beta'$ such that 
$T(A_1,\ldots,A_{s+\beta'}:s)=A_{i_1}\cup A_{i_2}\cup\cdots\cup A_{i_s}.$
As $|T(A_1,\ldots,A_{s+\beta'}:s)|=k$, $s$ must be equal to $1$.  Therefore  $|T(A_1,\ldots,A_{1+\beta'}:1)|=k$.
Since $T(A_1,\ldots,A_{1+\beta'}:1)=\cap_{i=1}^{1+\beta'}A_i$, we obtain
$\beta'=0$. As $t=2$, $s=1$, and $\F$ is $(s,t)$-union intersecting, there is a unique  $B_1\in\F'$ that $A_1\cap B_1=\varnothing.$ 
One can check that
$\F'\setminus \{B,B_1\}\subseteq \St_1(A_1,B:1)$.
Therefore, 
$|\F'|\leq {n-1\choose k-1}-{n- k-1\choose k-1}-{n- k-2\choose k-2}+2$, a contradiction.

\item $|T(A_1,\ldots,A_{s+\beta'}:s)|<\lfloor{(s+\beta)k\over \beta+1}\rfloor$. 

\noindent There is at most one member $B\in\F'$ such that $\F'\setminus\{B\}\subseteq \St_1.$
  Since $\F$ is an $(s,t)$-union intersecting family, every $s$ $k$-sets of $A_i$'s such as $A_{i_1},\ldots, A_{i_s}$ are disjoint from  at most $t-1$ $k$-subsets in $\F'$. Therefore, $$|\F'\setminus\{B\}|\leq |\St_1(A_1,\ldots,A_{s+\beta'}:s)|+{s+\beta'\choose s}(t-1).$$

 Now by applying Lemma~\ref{B:lem}~(b),
we obtain
 $$|\F'|\leq {n-1\choose k-1}-{n-|T(A_1,\ldots,A_{s+\beta'}:s)|-1\choose k-1}+O(n^{k-3})+{s+\beta'\choose s}(t-1)+1.$$
Since $|T(A_1,\ldots,A_{s+\beta'}:s)|<\lfloor{(s+\beta')k\over \beta'+1}\rfloor$ and $k\geq 3$, one can check that
 $$|\F|< {n-1\choose k-1}-{n-\lfloor{(s+\beta')k\over \beta'+1}\rfloor-1\choose k-1}$$
provided that $n$ is sufficiently large, which is not possible.
\end{enumerate}
\item[(ii)] $s+\hat{\beta}+1\leq\ell(\F)\leq M^{1-{1\over 3s}}.$

\noindent Let $\F'$ be a largest intersecting family of $\F$. Hence,  $|\F'|$ is at least  
$$ {n-1\choose k-1}-{n-\lfloor{(s+\beta)k\over \beta+1}\rfloor-1\choose k-1}-M^{1-{1\over 3s}}.$$
As $M=O(n^{k-2})$, we have $M^{1-{1\over 3s}}=o(n^{k-2})$.
Since $\lfloor{(s+\beta)k\over \beta+1}\rfloor\geq k$ and  $M^{1-{1\over 3s}}=o(n^{k-2})$, if $n$ is sufficiently large,
 then we have $$|\F'|> {n-1\choose k-1}-{n-k-1\choose k-1}-{n-k-2\choose k-2}+2.$$
By using Theorem~\ref{HK}, $\F'$ is a star or a Hilton-Milner family. Therefore, without loss of generality we may assume that there exists at most one  $B\in\F'$ 
such that $\F'\setminus\{B\}$ is a subfamily $\St_1$.

First assume that   $\lfloor{(s+\beta)k\over \beta+1}\rfloor\geq k+1$. 
By applying Lemma~\ref{B:lem}~(b) for $\F'\setminus\{B\}$ and one of $s+\hat{\beta}+1$ $k$-subsets of $\F\setminus\F'$, we obtain
$$|\F'|\leq {n-1\choose k-1}-{n-\lfloor{(s+\hat{\beta}+1)k\over \hat{\beta}+2}\rfloor-1\choose k-1}+{s+\hat{\beta}+1\choose s}(t-1)+1.$$ 
Hence, 
$$|\F|\leq {n-1\choose k-1}-{n-\lfloor{(s+\hat{\beta}+1)k\over \hat{\beta}+2}\rfloor-1\choose k-1}+{s+\hat{\beta}+1\choose s}(t-1)+M^{1-{1\over 3s}}+1.$$
Note that ${n-\lfloor{(s+\hat{\beta}+1)k\over \hat{\beta}+2}\rfloor-1\choose k-1}-{n-\lfloor{(s+\hat{\beta})k\over \hat{\beta}+1}\rfloor-1\choose k-1}={n-\lfloor{(s+\hat{\beta})k\over \hat{\beta}+1}\rfloor-1\choose k-2}$. Therefore, $|\F|$ is at most
$$ {n-1\choose k-1}-{n-\lfloor{(s+\hat{\beta})k\over \hat{\beta}+1}\rfloor-1\choose k-1}-{n-\lfloor{(s+\hat{\beta})k\over \hat{\beta}+1}\rfloor-1\choose k-2}+{s+\hat{\beta}+1\choose s}(t-1)+M^{1-{1\over 3s}}+1.$$
This concludes that for   sufficiently large $n$, $|\F|$ is less than  $M$, a contradiction.
 
 Assume that $\lfloor{(s+\beta)k\over \beta+1}\rfloor= k$. Therefore, $\hat{\beta}=\beta$ and  $\lfloor{(s+\beta+1)k\over \beta+2}\rfloor=k$. Take $A_1,\ldots,A_{s+\beta+1}$ in $\F\setminus \F'$. If we have 
$|T(A_1,\ldots,A_{s+\beta+1}:s)|<k$, then by applying Lemma~\ref{B:lem}~(b), we obtain that there exists a positive  constant $c$ such that
$$|\F'\setminus\{B\}|\leq {n-1\choose k-1}-{n-k\choose k-1}+cn^{k-3}+{s+\beta+1\choose s}(t-1).$$

This implies  $$|\F|\leq {n-1\choose k-1}-{n-k-1\choose k-1}-{n-k-1\choose k-2}+cn^{k-3}+{s+\beta+1\choose s}(t-1)+1+M^{1-{1\over 3s}},$$
which is less than $M$ when $n$ is sufficiently large, a contradiction.

 Assume that  $|T(A_1,\ldots,A_{s+\beta+1}:s)|=\lfloor{(s+\beta+1)k\over \beta+2}\rfloor= k$. In view of 
the last part of the proof of Lemma~\ref{B:lem}~(b), there  are $1\leq i_1<\ldots<i_s\leq s+\beta+1$ such that 
$$T(A_1,\ldots,A_{s+\beta+1}:s)=A_{i_1}\cup A_{i_2}\cup\cdots\cup A_{i_s}.$$
 This implies that  $s$ must be equal to $1$. If $s=1$, then we have $T(A_1,\ldots,A_{\beta+2}:1)=\cap_{i=1}^{\beta+2} A_i$ and 
 hence $ |T(A_1,\ldots,A_{\beta+2}:1)|=|\cap_{i=1}^{\beta+2} A_i|\leq k-1$ which contradicts with $|T(A_1,\ldots,A_{\beta+2}:1)|=k$. 

\item[(iii)]  $\ell(\F)> M^{1-{1\over 3s}}$.

\noindent By Lemma~\ref{lem:1}, we have $e(\KG_{n,k}[\F])\geq {M^{2-{2\over 3s}}\over 2{2k\choose k}}$
and by Theorem~\ref{bip}, $\F$ contains a subgraph which is isomorphic to $K_{s,t}$ when $n$ is sufficiently large. 
\end{enumerate}
\end{proof}
 Here we intend to  elaborate on the  $i$th largest $(s,t)$-union intersecting families for some $i$.
 Assume that $n$ is sufficiently large.
  Let $\{A_1, \ldots,A_s\}$ be $s$ pairwise distinct $k$-subsets of $[n]$. By Definition~\ref{def:1} we know that $T(A_1, \ldots,A_s:s)=\cup_{i=1}^{s}A_i$. 
 Define
  $$\LL\isdef\St_1(A_1,\ldots,A_s:s)\cup\{A_1,\ldots,A_s\}\cup\{F_1,\ldots, F_{t-1}\} $$ where  $F_i\in\St_1\setminus\St_1(A_1,\ldots,A_s:s).$
By using Inequality~(\ref{eq2}), one can verify that $|\LL|$ is equal to $${n-1\choose k-1}-{n-|T(A_1, \ldots,A_s:s)|-1\choose k-1}+s+t-1.$$
Let $n=n(k,s)$ be sufficiently large and $s\geq 2$. 
If  $ \lfloor{(s+1)k\over 2}\rfloor<|T(A_1, \ldots,A_s:s)|\leq sk$, then by using Theorem~\ref{HMnew3}, 
$\LL$ is the $i$th largest $(s,t)$-union intersecting  family,
 where  $i=sk-|T(A_1, \ldots,A_s:s)|+2$.

If  $ |T(A_1, \ldots,A_s:s)|=\lfloor{(s+1)k\over 2}\rfloor$, then $|\LL|$ is equal to ${n-1\choose k-1}-{n-\lfloor{(s+1)k\over 2}\rfloor-1\choose k-1}+s+t-1$.
Let  $\{A'_1, \ldots,A'_s,A'_{s+1}\}$ be $s+1$ pairwise distinct $k$-subsets of $[n]$ such that $T(A'_1, \ldots,A'_{s+1}:s)=\lfloor{(s+1)k\over 2}\rfloor$. 
 Define $$\LL'\isdef\St_1(A'_1,\ldots,A'_{s+1}:s)\cup\{A'_1,\ldots,A'_{s+1}\}\cup\{F'_1,\ldots, F'_{t-1}\}.$$ We have $|\LL'|$
  is equal to ${n-1\choose k-1}-{n-\lfloor{(s+1)k\over 2}\rfloor-1\choose k-1}+s+t$ which is greater than $|\LL|$. Therefore, $\LL'$ and $\LL$
  are the $(\lfloor{(s-1)k\over 2}\rfloor+2)$th and $(\lfloor{(s-1)k\over 2}\rfloor+3)$th largest $(s,t)$-union intersecting  families, respectively.
  
  Now assume that there are $\{A_1, \ldots,A_s\}$ and $\{A'_1, \ldots,A'_{s+1}\}$ such that  $ |T(A_1, \ldots,A_s:s)|=\lfloor{(s+1)k\over 2}\rfloor-1$
  and $ |T(A'_1, \ldots,A'_{s+1}:s)|=\lfloor{(s+1)k\over 2}\rfloor-1$. If $(s+1)k$ is even, then $2 |T(A'_1, \ldots,A'_{s+1}:s)|=(s+1)k-2$.
  Therefore, there are at most two members \ in $\cup_{i=1}^{s+1}A'_i$ such that each of them appears in one of $A'_i$'s. If for each $i\leq s+1$
we have $A'_i\subset T(A'_1, \ldots,A'_{s+1}:s)$, then by using Inequality~(\ref{eq2}),
 we have 
 $$|\LL'|={n-1\choose k-1}-{n-\lfloor{(s+1)k\over 2}\rfloor\choose k-1}+s+t.$$  If for only one $i\leq s+1$ we have
$A'_i\not\subseteq  T(A'_1, \ldots,A'_{s+1}:s)$, then one can construct an $(s,t)$-union intersecting family $\LL'_1$ with
$\ell(\LL'_1)=s+1$ and
 $$|\LL'_1|={n-1\choose k-1}-{n-\lfloor{(s+1)k\over 2}\rfloor\choose k-1}+s+t.$$  
  Now suppose that  $A'_i\not\subseteq T(A'_1, \ldots,A'_{s+1}:s)$ and 
   $A'_j\not\subseteq T(A'_1, \ldots,A'_{s+1}:s)$ for exactly two $1\leq i\neq j\leq s+1$. By using Inequality~\ref{eq2},
 the number of $A\in \St_1$ 
 which has no common element with $T(A'_1, \ldots,A'_{s+1}:s)$ and intersects  at least two of $A'_i$'s
  is ${n-|T(A'_1, \ldots,A'_{s+1}:s)|-3\choose k-3}$. Therefore,  for $0\leq m\leq t-1$, one can construct a maximal $(s,t)$-union family $\LL'_{2,m}$ with
  $\ell(\LL'_{2,m})=s+1$ and  $$|\LL'_{2,m}|={n-1\choose k-1}-{n-\lfloor{(s+1)k\over 2}\rfloor\choose k-1}+{n-\lfloor{(s+1)k\over 2}\rfloor-2\choose k-3}+s+t+m.$$  
Therefore, we have some different types  $(s,t)$-union intersecting families with $\ell(\F)=s+1$,  $|T(A'_1, \ldots,A'_{s+1}:s)|=\lfloor{(s+1)k\over 2}\rfloor-1$, and different sizes and one type   of  $(s,t)$-union intersecting families  with $\ell(\F)=s$,  $|T(A_1, \ldots,A_{s}:s)|=\lfloor{(s+1)k\over 2}\rfloor-1$.

 If $(s+1)k$ is odd, then $2 |T(A'_1, \ldots,A'_{s+1}:s)|=(s+1)k-3$. Therefore, there are at most three members in $\cup_{i=1}^{s+1}A'_i$ such that each of them appears in one of $A'_i$'s.  Using the same discussion as above one can find some different types
 of    $(s,t)$-union intersecting families with
  $\ell(\F)=s+1$,  $|T(A'_1, \ldots,A'_{s+1}:s)|=\lfloor{(s+1)k\over 2}\rfloor-1$ and different sizes.

In the proof of Theorem~\ref{HMnew2}, we need  the following theorem by Frankl~\cite{MR670845} and independently Kalai~\cite{MR770699} which is a generalization of a  classical result due to Bollob{\'a}s~\cite{MR0183653}. 
\begin{alphtheorem}\label{bl}{\rm \cite{MR670845,MR770699}}
Let $\{(A_1,B_1),\ldots,(A_h,B_h)\}$ be a family of pairs of subsets of an arbitrary set with $|A_i|=k$ and $|B_i|=\ell$ for all $1\leq i\leq h$.
If $A_i\cap B_i=\varnothing$ for  $1\leq i\leq h$ and $A_i\cap B_j\neq\varnothing$ for  $1\leq i< j\leq h$, then
$h\leq {k+\ell\choose k}$.
\end{alphtheorem}
For  simplicity of notation, for each $1\leq i\leq k-1$, define $N_{i}\isdef{n-1\choose k-1}-{n-k\choose k-1}+{n-k-i\choose k-i-1}$ and 
 for  $k$ define $N_{k}\isdef{n-1\choose k-1}-{n-k\choose k-1}$. 
Note that for  $1\leq i\leq k-1$, we have $N_{i-1}-N_{i}={n-k-i\choose k-i}=\Omega(n^{k-i}).$
\begin{proof}[Proof of Theorem~{\rm\ref{HMnew2}}]
First we show that $\ell(\F)\leq {2k-1\choose k-1}(t-1)$. 
  If $t=1$, $\F$ is intersecting  and hence $\ell(\F)=0$.  Assume that $t\geq 2$ and $\F$ is not intersecting.
Therefore, there exists some disjoint 
pair in $\F$. For a $k$-set $A$, define $N(A)=\{B\in{[n]\choose k}|\, A\cap B=\varnothing \}$.
Define $\F_1=\F$.
For each $i\geq 2$, if there exists some disjoint 
pair in $\F_{i-1}$, 
choose $B_{i-1}\in \F_{i-1}$ and $C_{i-1}\in N(B_{i-1})\cap \F_{i-1}$
 and define $\F_i=\F_{i-1}\setminus (N(B_{i-1})$.  Let $m$ be the largest index $i$ for which $\F_i$ contains some disjoint 
pair. 
For  $m+1\leq j\le 2m$, set $B_j=C_{2m-j+1}$ and $C_j=B_{2m-j+1}$. 
One  may check that the family$\{(B_1,C_1),\ldots,(B_{2m,}C_{2m})\}$
satisfies the condition of Theorem~\ref{bl} for $l=k$ and consequently  $m\leq {2k-1\choose k-1}$. 
Let $\N$ be a subfamily of $\F$ defined as follows
 $$\N=\Big\{F\in\F| {\rm there\,\, is\,\, some\,\,} i\leq m\,\,{\rm such\,\, that}\,\, F\cap B_i=\varnothing \Big\}.$$ 
 Since $\F$ is $(1,t)$-union intersecting, one can verify that $|\N|\le m(t-1)$. Note that $\F_{m+1}$ is an intersecting family and $\F$ is disjoint union of
 $\F_{m+1}$ and $\N$. This yields $\ell(\F)\leq |\N|\leq {2k-1\choose k-1}(t-1)$.


Assume that $|\F|=N_{\gamma}+\gamma t$. 
Let $\F^{*}$ be one of largest intersecting subfamilies of $\F$ such that  $\Delta(\F^*)$ has the maximum possible value.
 Assume that  $\F\setminus\F^{*}=\{A_1,\ldots,A_{\ell(\F)}\}$. 
Therefore, 
$|\F^*|= |\F|-\ell(\F).$
Consider the following three cases.
\begin{enumerate}
\item {\it $\ell(\F)=\gamma$ and $\F^*\subseteq\St_1$. }

\noindent We have $|\F^{*}|=N_{\gamma}+\gamma(t-1)$.
Since   $\ell(\F)=\gamma$ and $\F=\F^{*}\cup\{A_1,\ldots A_\gamma\}$, each $ A_j$ is disjoint from at least one member of $\F^*$ and hence 
 $1\not\in\cup_{j=1}^{\gamma} A_j$. Then
  $$\F^{*}\setminus(\cup_{j=1}^{\gamma} N(A_j))\subseteq \St_1(A_1,\ldots,A_{\gamma}:1).$$ 
  Since $\gamma\leq k-2$, 
   by applying Lemma~\ref{B:lem}~(c), we conclude that 
$|\F^{*}\setminus(\cup_{j=1}^{\gamma} N(A_j))|\leq N_{\gamma}$.
Since $\F$ is $(1,t)$-union intersecting,  for each $j$, $A_j$ is  disjoint from at most $t-1$ members of $\F$. 
As for each $j$, $|N(A_j)\cap \F|\leq t-1$, 
$|\F|=N_{\gamma}+\gamma t$, and
$$\F=\F^{*}\setminus(\cup_{j=1}^{\gamma} N(A_j))\cup(\cup_{j=1}^{\gamma} N(A_j)\cap\F)\cup\{A_1,\ldots,A_{\gamma}\},$$
we have $\F$ is a disjoint union of
$$\F^{*}\setminus(\cup_{j=1}^{\gamma} N(A_j)),N(A_1)\cap \F,\ldots, N(A_{\gamma})\cap \F,\,\, {\rm and}\,\, \{A_1,\ldots,A_{\gamma}\}.$$ Moroever, 
for each $j$, we have $|N(A_j)\cap \F|=t-1$, $N(A_j)\cap \F\subseteq\F^*\subseteq \St_1$, and  $|\F^{*}\setminus(\cup_{j=1}^{\gamma} N(A_j))|=N_{\gamma}$.
From the last equality and by using Lemma~\ref{B:lem}~(c), we obtain
  $$\F^{*}\setminus(\cup_{j=1}^{\gamma} N(A_j))=  \St_1(A_1,\ldots,A_{\gamma}:1)$$
and $|\cap_{j=1}^{\gamma} A_j|=k-1.$ By taking $E=\cap_{j=1}^{\gamma} A_j$ and $J=\{1\}\cup(\cup_{j=1}^{\gamma} A_j\setminus E)$  in Definition~\ref{def1},
one can see that $\F\setminus(\cup_{j=1}^{\gamma} N(A_j))$ is isomorphic to  $\J_{\gamma}$.
For each $j\leq \beta+1$, by taking $\B_{j}=N(A_j)\cap \F$  in Definition~\ref{def2},
one can check that $\F$ is isomorphic to  $\J^{1,t}_{\gamma}$.
By Theorem~\ref{Kup}, $\F^{*}$ is either a star or isomorphic to a subfamily  $\J_i$ where $0\leq i\leq \gamma-1$. 
First let $\F^{*}$  be a star and $\F^{*}\subseteq \St_1$. 
\item{\it $\gamma+1\leq\ell(\F)\leq {2k-1\choose k-1}(t-1)$ and $\F^*\subseteq\St_1$. } 

\noindent Let $A_1,\ldots,A_{\gamma+1}\in \F\setminus\F^*.$ By using minimality of $\ell(\F)$,
 each $A_i$ is disjoint from at least one member of $\F^*$. Therefore, $1\not\in A_i$ for each $i\leq \gamma+1.$ 
  Then
    $$\F^{*}\setminus((\cup_{i=1}^{\gamma+1} N(A_i))\subseteq
     \St_1(A_1,\ldots,A_{\gamma+1}:1)$$
  and by applying Lemma~\ref{B:lem}~(c), we obtain $|\F^{*}\setminus((\cup_{i=1}^{\gamma+1} N(A_i))|\leq N_{\gamma+1}$.
  Since $$\F=(\F^{*}\setminus (\cup_{j=1}^{\gamma+1} N(A_j))\cup (\cup_{i=1}^{\gamma+1}N(A_i)\cap\F)\cup\{A_1,\ldots,A_{\ell(\F)}\},$$
  we have $|\F|\leq N_{\gamma+1}+(\gamma+1)(t-1)+\ell(\F)<N_{\gamma}$,  which is not possible 
when $n$ is sufficiently large.
\item{\it  $\gamma\leq\ell(\F)\leq {2k-1\choose k-1}(t-1)$  and $\F^*$ is not  a star. }

\noindent By Theorem~\ref{Kup}, $\F^*\subseteq \J_{c}$ for  some $1\leq c\leq \beta+1$. Then, for some $b\leq c$,
  there exist $B_1,\ldots,B_{b}\in\F^*$ such that   $\F^*\setminus\{B_1,\ldots, B_{b}\}\subseteq\St_1$ and $B_j\not\in\St_1$.
At most $b-1$ of $A_1,\ldots,A_\gamma$  contain $1$;  otherwise if for $1\leq j_1\leq \cdots\leq j_b\leq  \gamma$ we have
 $1\in\cap_{i=1}^{b}A_{j_i}$,
then $\F'=(\F^*\setminus\{B_1,\ldots,B_{b}\})\cup\{A_{j_1},\ldots,A_{j_b}\}$
 is an intersecting family  with $|\F'|=|\F^*|$ and $\Delta(\F')>\Delta(\F^*)$, 
which  contradicts with  the fact that   $\Delta(\F^*)$ 
has the maximum possible value.
Therefore, without loss of generality
 we can assume that    $A_1,\ldots,A_{b'}$
 do not contain $1$  for $b'=\gamma+1-b$. Hence, 
  $$\F^{*}\setminus((\cup_{j=1}^{b'} N(A_j))\cup\{B_1,\ldots, B_{b}\})\subseteq \St_1(A_1,\ldots,A_{b'},B_1,\ldots, B_{b}:1)$$
  and by  Lemma~\ref{B:lem}~(c), we obtain $|\F^{*}\setminus((\cup_{j=1}^{b'} N(A_j))\cup\{B_1,\ldots, B_{b}\})|\leq N_{\gamma+1}$.
  Since $$\F=(\F^{*}\setminus \cup_{j=1}^{b'} N(A_j))\cup (\cup_{j=1}^{b'}N(A_j)\cap\F)\cup\{A_1,\ldots,A_{\ell(\F)}\},$$
  we obtain $|\F|\leq N_{\gamma+1}+b+b'(t-1)+\ell(\F)<N_{\gamma}$, which is not possible when $n$ is sufficiently large. 
\end{enumerate}
\end{proof}

For the proof of Theorem~\ref{last} we need to use the well-known Erd{\H o}s-Stone-Simonovits theorem~\cite{MR0205876,MR0018807}. 
For a given graph $G$, the {\it Tur\'an number ${\rm ex}(n,G)$}
 is defined to be the maximum number of edges in a graph with $n$ vertices containing no subgraph isomorphic to $G$.  
The   Erd{\H o}s-Stone-Simonovits theorem asserts that for any graph $G$ with $\chi(G)\geq 2$,
${\rm ex}(G,n)=(1-{1\over \chi(G)-1}){n\choose 2}+o(n^2).$

\begin{proof}[Proof of Theorem~{\rm\ref{last}}]
The proof is by induction on $r$. By Theorem~\ref{HMnew3}, the assertion is true when $r=1$. Let $r\geq 2$. Suppose now that the assertion is true for $r-1$. Also, without loss of generality suppose that
 $$|\F|= {n\choose k}-{n-r\choose k}-{n-\lfloor{(s_{r+1}+{\beta})k\over {\beta}+1}\rfloor-r\choose k-1}+s_{r}+s_{r+1}+\hat{\beta}-1.$$
Consider the following cases.

\begin{enumerate}
\item $\max_{i\in[n]} |\F\cap\St_i|\leq{n-1\choose k-1}-{n-\sum_{j=2}^{r+1}s_{j}k-1\choose k-1}+s_{1}$. 

\noindent Then the number of disjoint pair in $\F$ is at least
$${|\F|\choose 2}-\sum_{i\in [n]}{|\F\cap \St_i|\choose 2}\geq (1-{1\over r}){|\F|\choose 2}+o(|\F|^2)$$
provided that $n$ is large enough. Hence, by the Erd{\H o}s-Stone-Simonovits theorem $\KG_{n,k}[\F]$ contains some subgraph isomorphic to $K_{s_1,s_2,\ldots,s_{r+1}}$ provided that n is large enough, which is a contradiction.

\item $\max_{i\in[n]} |\F\cap\St_i|> {n-1\choose k-1}-{n-\sum_{j=2}^{r+1}s_{j}k-1\choose k-1}+s_{1}$.

\noindent  Without loss of generality assume that 
$\max_{i\in[n]} |\F\cap\St_i|=|\F\cap\St_n|$.  Notice that one can assume that $\St_n\subset \F$. Because otherwise, if $\St_n\not\subset \F$, then $|\F\cap\St_n|< {n-1\choose k-1}$. Therefore, 
$$|\F\setminus \St_{n}|\geq {n-1\choose k}-{n-r\choose k}-{n-\lfloor{(s_{r+1}+{\beta})k\over {\beta}+1}\rfloor-r\choose k-1}+s_{r}+s_{r+1}+\hat{\beta}.$$
By induction hypothesis $\KG_{n-1,k}[\F\setminus \St_{n}]$ contains $K_{s_2,\ldots,s_{r+1}}$. As $$ |\F\cap\St_n|> {n-1\choose k-1}-{n-\sum_{j=2}^{r+1}s_{j}k-1\choose k-1}+s_{1},$$
one can greedily pick $s_1$ sets of $\St_n$ such that constructs a copy of $K_{s_1,s_2,\ldots,s_{r+1}}$ in $\KG_{n,k}[F]$, a contradiction.
Since  $\St_n$ is a subfamily of $\F$,  same as previous discussion  $K_{s_2,\ldots,s_{r+1}}$ cannot be  contained in  $\KG_{n-1,k}[\F\setminus \St_{n}]$. 
Therefore, by induction hypothesis,
we have
 $$|\F\setminus \St_{n}|\leq {n-1\choose k}-{n-r\choose k}-{n-\lfloor{(s_{r+1}+{\beta})k\over {\beta}+1}\rfloor-r\choose k-1}+s_{r}+s_{r+1}+\hat{\beta}-1,$$ 
 and the equality holds if and only if $\F\setminus \St_{n}$ is isomorphic to
$$\bigcup_{i=1}^{r-2}(\St_i\setminus\St_n)\cup(\St_{r-1}^{[r-2]}(A_1, A_2,\ldots,A_{s_{r+1}+\hat{\beta}}:s)\setminus\St_n)\cup\{A_1, A_2,\ldots,A_{s_{r+1}+\hat{\beta}}\}\cup\{F_1,\ldots, F_{s_r-1}\}$$ 
such that 
$|T(A_1, A_2,\ldots,A_{s_{r+1}+\hat{\beta}})|=\lfloor{(s_{r+1}+\beta)k\over \beta+1}\rfloor$,
$F_i\in\St_{r-1}\setminus\St_{r-1}^{[r-2]}(A_1, A_2,\ldots,A_{s+\hat{\beta}}:s)$, and $F_i\cap [r-2]=\varnothing$ for each $i$ ( Note that in this step all families are subfamilies of 
${[n-1]\choose k}$ because we remove $\St_n$ from $\F$ so we do not meet $n$.).

Thus,  $$|\F|\leq {n\choose k}-{n-r\choose k}-{n-\lfloor{(s_{r+1}+{\beta})k\over {\beta}+1}\rfloor-r\choose k-1}+s_{r}+s_{r+1}+\hat{\beta}-1,$$ 
 and the equality holds if and only if $\F$ is isomorphic to
$$\bigcup_{i=1}^{r-1}\St_i\cup\St_r^{[r-1]}(A_1, A_2,\ldots,A_{s_{r+1}+\hat{\beta}}:s)\cup\{A_1, A_2,\ldots,A_{s_{r+1}+\hat{\beta}}\}\cup\{F_1,\ldots, F_{s_r-1}\}$$ such that 
$|T(A_1, A_2,\ldots,A_{s_{r+1}+\hat{\beta}})|=\lfloor{(s_{r+1}+\beta)k\over \beta+1}\rfloor$,
$F_i\in\St_r\setminus\St_r^{[r-1]}(A_1, A_2,\ldots,A_{s+\hat{\beta}}:s)$, and $F_i\cap [r-1]=\varnothing$ for each $i$.
\end{enumerate}

\end{proof}
Proofs of Theorem~\ref{last1}  and Corollary~\ref{lastcor} are the same as the proof of Theorem~\ref{last}.
\section*{Acknowledgements}
The author is grateful to Meysam Alishahi and  Amir Daneshgar for their valuable comments. This research  was in part supported by a grant from IPM (No. 98050012).

\bibliographystyle{plain}

\end{document}